\documentclass[smallextended]{svjour3}     

\usepackage{amsmath,amssymb,bm,framed,bbm}
\usepackage{color,graphicx,hyperref}
\usepackage[numbers]{natbib}
\smartqed

\usepackage{a4wide}
\begin{document}
%\begin{frontmatter}

\title{Shot noise, weak convergence and diffusion approximations}

\author{Massimiliano Tamborrino \and Petr Lansky}
\authorrunning{Tamborrino, Lansky} % if too long for running head

\institute{M. Tamborrino \at Department of Statistics, University of Warwick, Gibbet Hill Road, CV4 7AL, Coventry United Kingdom\\
\email{massimiliano.tamborrino@warwick.ac.uk}
\and P. Lansky \at Institute of Physiology, Academy of Sciences of the Czech Republic,Videnska 1083,
Prague 4, 142 20, Czech Republic\\
\email{petr.lansky@fgu.cas.cz}}
\date{}%Received: date / Accepted: date}
\maketitle

 \begin{abstract} 
Shot noise processes have been extensively studied due to their mathematical properties and their relevance in several applications.  Here, we consider nonnegative shot noise processes and prove their weak convergence
to L{\'e}vy-driven Ornstein-Uhlenbeck (OU), whose features depend on the underlying jump distributions. {Among others, we obtain the OU-Gamma and OU-Inverse Gaussian processes, having gamma and inverse gaussian processes as background L\'evy processes, respectively.} Then, we derive the necessary conditions guaranteeing the diffusion limit to a Gaussian OU process, show that they are not met unless allowing {for negative jumps} happening with probability going to zero, and quantify the error occurred when replacing the shot noise with the OU process {and the non-Gaussian OU processes}. The results offer a new class of models to be used instead of the commonly applied Gaussian OU processes to approximate synaptic input currents, membrane voltages or conductances modelled by shot noise in single neuron modelling. 
\end{abstract}

\section{Introduction}
Shot noise processes, initially proposed to model shot noise in vacuum tubes \cite{Schottky1918}, {have been generalized} and used to model various phenomena in several areas of applications. An incomplete list includes anomalous diffusion in physics, earthquakes occurrences in geology, rainfall modelling in meteorology, network traffic in computer sciences, insurance, finance and neuroscience; see \cite{Iksanovetal2017} and
references therein for an exhaustive overview. Here, we focus on neuroscience, and in particular on single neuron modelling, where shot noise processes known as Stein's neuronal model \cite{Stein1965} {without inhibition were initially proposed} to describe membrane voltage in the framework of Leaky Integrate-and-Fire models \cite{TuckwellBook,GerstnerKistler2002}. More recently, shot noise processes have been used to describe synaptic input currents in Integrate-and-Fire neurons \cite{DL} and Leaky Integrate-and-Fire models \cite{Olmietal2017,HohnBurkitt,Iyeretal2013}, as well as   conductances in conductance-based models \cite{RichardsonGerstnerChaos}.  Since the mathematical formulation of the shot noise process is the same, the results derived in this paper do not depend on the underlying modelling framework, and can therefore be applied to all considered scenarios. 

Suppose that events (e.g. jumps representing the excitatory postsynaptic potentials) occur according to a homogeneous Poisson process $N$ with constant rate $\lambda>0$. Associated with the $k$th event is a positive random variable $J^{[k]}$,
quantifying the nonnegative random jump amplitude of $k$th input. Denote by $\tau_k$ the time of the $k$th jump. Assume jumps {$J^{[k]}$} to be identically distributed,  independent of each other and of the point process. Then, the Poissonian shot noise process $\{X(t),t\geq 0\}$, or simply {\em shot noise process}, is the resulting superposition
\begin{equation}\label{1}
X(t)=x_0e^{-\alpha t}+\sum_{k=1}^{N(t)} {J^{[k]}} e^{-\alpha(t-\tau_k)}, \qquad X(0)=x_0\geq 0,
\end{equation}
where $\alpha>0$ is a constant determining the exponential decay rate, e.g. the
inverse of the membrane time constant when $X(t)$ models the membrane voltage of a neuron. In general, when the distribution of $J$ assumes real values, the model is known as Stein's model \cite{Stein1965} in the neuroscience literature.

Limiting results have been proposed for the drifted, rescaled in time, real {valued} and normalized shot noise process or its generalised versions, see e.g. \cite{Iksanovetal2017,PangZhou2018}. Here, instead,   
 we are interested in deriving limiting results for the original nonnegative shot noise process \eqref{1}, since this process,  and {not its variants}, plays an important role in Leaky Integrate-and-Fire and conductance based models.  From a mathematical point of view, this is commonly done by investigating a sequence of jump processes $(X_n)_{n\in\mathbb{N}}=(\{X_n(t),t\geq 0\})_{n\in\mathbb{N}}$ for which the distribution of the trajectories gets close to that of a {limiting} process $Y$. 
{Depending on the assumptions} on the frequencies and on the distributions of the jumps of the underlying stochastic point process, the limit process can be either deterministic, obtained, e.g., as a solution of systems of ordinary/partial differential equations \cite{Kurtz1970,Kurtz1971,PTW,RTW,SL2017}, or stochastic {\cite{SL2017}},\cite{BilConv,Jacod,Ricciardi}. Limits of the first type are called fluid limits, thermodynamic limits, hydrodynamic limits or Langevin approximations, and give rise to what is also known as Kurtz approximation \cite{Kurtz1971}. In this setting,  the frequency of jumps is {assumed} roughly inversely proportional to the jump size and the noise term is {assumed} proportional to ${1}/{\sqrt{n}}$, so the approximation of a jump with a diffusion process holds for fixed $n$, leading to a deterministic process as $n\to\infty$.
 
In this paper, we focus on limits  of the second type, 
dealing with {\em weak convergence} of stochastic processes \cite{BilConv,Jacod,GikhmanSkorokhod}, illustrated, e.g., in \cite{LanskyLanska1987,TSJ2014} in the neuronal context. In this setting, jump amplitudes {are assumed to decrease} to zero for $n\to\infty$, but occur at an increasing frequency roughly inversely proportional to the square of the jump size. Limiting results are proved by showing the convergence of triplets of characteristics, as proposed in \cite{Jacod} and illustrated, e.g., in \cite{TSJ2014}. An alternative proof for one dimensional processes converging to diffusion processes is that provided by Gikhman and Skorokhod \cite{GikhmanSkorokhod}. 
The {necessary conditions for the convergence of jump processes to diffusion processes}  are that the limits of the first two infinitesimal moments (also known as Kramers-Moyal coefficients) of the jump process converge to those of the {diffusion} process, and that the fourth infinitesimal moment goes to zero \cite{GikhmanSkorokhod,KarlinTaylor}. The vanishing of the fourth infinitesimal moment guarantees that the limit process is continuous \cite{Pawula}, a necessary condition for obtaining {diffusion processes}. In mathematical neuroscience, such approach was often studied by Ricciardi  and his coworkers \cite{CapocelliRicciardi1971,Ricciardi1976}, and has been referred to as {\em diffusion approximation} \cite{TuckwellBook,Ricciardi}. {A problem arises} for the use of several notions of diffusion approximations. {For example, the previously mentioned Langevin or Kurtz approximation, leading to a deterministic process as $n\to\infty$, is also known as diffusion approximation.}  {An alternative concept of diffusion approximation, called {\em usual approximation} or truncated Kramers-Moyal expansion, was for discontinuous models {with relatively small jumps} \cite{Walsh}. 
 This method replaces the discontinuous process $X$ with a diffusion process $Y$ having the same first two infinitesimal moments.} Another notion, which we will refer to as {\em Gaussian approximation} or matched synaptic distributions \cite{Iyeretal2013,RichardsonGerstnerChaos}, consists in replacing the discontinuous process $X$ with a Gaussian process $Y$  having the same first two infinitesimal moments. Since these two notions involve neither limiting results nor infinitesimal moments of order higher than two, the limit process may also be discontinuous,  yielding thus a low approximation accuracy,  as shown in Section 4.2.3. 
Over the years, the rigorous approach of diffusion approximation by Ricciardi, mathematically supported by the findings of  Gikhman and Skorokhod \cite{GikhmanSkorokhod}, has been replaced by the {usual and Gaussian approximations} in the mathematical and computational neuroscience community, see e.g. \cite{DL,RichardsonGerstnerChaos,SchwalgerDrosteLindner}. 

The goals of {this paper} is to perform weak convergence of a sequence of nonnegative shot noise processes \eqref{1} and investigate if, and under which conditions for the amplitudes and frequencies of the jumps, a sequence of them  admits a {diffusion process} as limit.  First, we show that the obtained limit process is a (discontinuous) non-Gaussian OU process,  {also known as L{\'e}vy-driven OU} or OU L{\'e}vy process \cite{levyOU}, i.e.  an OU process having a non-Gaussian L{\'e}vy process as driving noise,  as the initial shot noise process. Then, we characterise the limiting L{\'ev}y measures which depend on the jump amplitude distributions of the shot noise, {showing that the OU-Poisson, OU-Gamma, OU-Inverse Gaussian  and OU-Beta process can {be obtained as} limit processes having a Poisson, gamma, inverse gaussian and beta process as background driving L\'evy processes, respectively.}
 We refer to \cite{Bertoin1996,Sato1999} as standard references on L{\'e}vy processes. Moreover, we derive the necessary conditions to {perform a diffusion approximation} and show that these are not simultaneously met, as expected by the non negativity of the jumps. Hence,  the Gaussian OU process cannot be obtained as a diffusion approximation of the shot noise process, {but only as a usual or Gaussian approximation} \cite{DL,RichardsonGerstnerChaos,MelansonLongtin2019}. However, we prove that {modifying} the assumption of nonnegative jumps and allowing for jumps with negative amplitudes happening with probability going to 0 is enough to guarantee the weak convergence to a Gaussian OU process as diffusion limit.

On one hand, our results show {how the different limit approaches and notions of diffusion approximation may lead to different approximating models. {In particular, since the usual and Gaussian approximations do not check the behaviour of the fourth infinitesimal moment, they implicitly assume that the limit process is continuous, while instead the non vanishing of the fourth infinitesimal moment results in a discontinuous limit process.}
  On the other hand, they may be used to improve existing results on single neuronal models and their corresponding firing statistics, by replacing the membrane voltage, the synaptic input currents or the conductances modelled by the shot noise with the obtained L{\'evy}-driven  (and not Gaussian) OU processes,   {as illustrated here. The combination of deriving explicit expressions for the conditional mean,  variance and characteristic function,  the possibility to simulate these processes exactly via a provided package written in the computing environment R \cite{R} and their outperformance over the OU process as approximating models for the shot noise makes the derived non-Gaussian OU processes a powerful and tractable class of models.} {Our findings} are not specific for the shot noise process, but can be directly applied to all those models where a diffusion approximation involving sequences of nonnegative and/or nonpositive random variables is  {commonly used}, e.g.  neuronal models with synaptic reversal potentials \cite{LanskyLanska1987}, as
previously observed in \cite{Cupera2014}.  
 
\section{Poissonian shot noise}
{Unless otherwise specified}, we consider a {nonnegative} shot noise process \eqref{1} whose random jump amplitudes {$J^{[k]}, k\in\mathbb{N}$} are independent and identically distributed {nonnegative} random variables, independent on the Poisson process $N(t)$. We denote by $J$ the {amplitude} of the jumps, with cumulative distribution function (cdf) $F_J$ {and probability mass function (pmf)/probability density function (pdf) $f_J$}. 
The shot noise process \eqref{1} is obtained as the solution of the stochastic differential equation
\begin{equation}\label{sde}
dX(t)=-\alpha X(t) dt + JdN(t), \qquad X(0)=x_0\geq 0,
\end{equation}
and has state space $[0,\infty)$. {Denote by}
\[
L(t)=\sum_{k=1}^{N(t)}{J^{[k]}},
\]
a compound Poisson process with $L(0)\equiv 0$, i.e. a L{\'e}vy process with finite (bounded) L{\'e}vy measure {$\lambda F_J(dx)$}, drift 0 and no diffusion component (the Gaussian part). {Thanks to the L\'evy-Khintchine decomposition \cite{Bertoin1996,Sato1999}}, the characteristic function of $L(t)$ {is} {\cite{Bertoin1996,Sato1999}}
\begin{equation}\label{charL}
\mathbb{E}[e^{iuL(t)}]=\exp\left(t \int_{{0}}^{{\infty}} (e^{iux}-1)\lambda F_J(dx)\right),
\end{equation}
while that of $X_t$ is  {\cite{Bertoin1996,Sato1999}}
\begin{equation}\label{mgf}
\mathbb{E}[e^{iuX(t)}]=\exp\left(\int_{{0}}^{{\infty}} (e^{iu xe^{-\alpha y}}-1)dy \lambda F_J(dx)+iux_0e^{-\alpha t}\right),
\end{equation}
{These expressions can be rewritten using densities instead of measures, using the fact that $\lambda F_j(dx)=\lambda f_J(x)dx$.} 
The mean,  covariance and variance of the shot noise process are given by \cite{Ross}
\begin{eqnarray}
\label{mom1}\mathbb{E}[X(t)]&=&\frac{\lambda}{\alpha}\mathbb{E}[J](1-e^{-\alpha t})+ x_0e^{-\alpha t},\\
\label{mom3}\textrm{Cov}(X(t),X(t+s))&=&\frac{\lambda}{2\alpha}e^{-\alpha s}\mathbb{E}[J^2](1-e^{-2\alpha t}),\quad {s>0,}\\
\label{mom2}\textrm{Var}(X(t))&=&\frac{\lambda}{2\alpha}\mathbb{E}[J^2](1-e^{-2\alpha t}).
\end{eqnarray}
The fourth moment of $X$ can be obtained from the characteristic function, yielding
\begin{eqnarray*}
\mathbb{E}[X^4(t)]&=&\mathbb{E}^4[X(t)]+6\mathbb{E}^2[X(t)]\textrm{Var}(X(t))+3\textrm{Var}^2(X(t))
+\frac{4\lambda}{3\alpha}\mathbb{E}[X(t)]{\mathbb{E}[J^3]}(1-e^{-3\alpha t})\\
&+& \frac{\lambda}{4\alpha}\mathbb{E}[J^4](1-e^{-4\alpha t}).
\end{eqnarray*}

\section{{L{\'e}vy-driven OU process as limit model of the {nonnegative} shot noise}}\label{Section4}
Let us consider a sequence of compound Poisson processes $(L_n)_{n\in\mathbb{N}}$ with jump distribution $F_{J_n}$ under  the assumptions that 
\begin{equation}\label{lambda}
\lim_{n\to\infty}\lambda_n =\infty
\end{equation}
and
 \begin{equation}\label{weak}
{\lambda_n F_{J_n}\stackrel{\mathcal{L}}{\to}\nu},
\end{equation}
where $\nu$ is an unbounded L{\'e}vy measure and the convergence in law $\mathcal{L}$ in \eqref{weak} is defined such that
\[
\int h(x)\lambda_nF_{J_n}(dx)\to \int h(x)\nu(dx), \qquad n\to\infty,
\]
for all bounded and continuous functions $h$, differentiable at $x=0$ with $h(0)=0$. {When the L\'evy measure can be rewritten as $\nu(dx)=u(x)dx$, where $u(x)$ is the L\'evy density also known as L\'evy-Khintchine density  \cite{BarndorffNielsenShephard}, condition \eqref{weak} can be rewritten in terms of densities as
\[
\lambda_n f_{J_n}\stackrel{\mathcal{L}}{\to}u.
\]
}

Then, we have the following
\begin{theorem}\label{Theo1}
Under conditions \eqref{lambda} and \eqref{weak}, the sequence of shot noise processes $(X_n)_{n\in\mathbb{N}}$  converges weakly to a nonnegative L{\'e}vy-driven OU process $Y=\{Y(t),t\geq 0\}$ given as the solution of the stochastic differential equation \begin{equation}\label{Levy}
dY(t)=-\alpha Y(t)dt + dL_\infty(t), \qquad Y(0)=x_0\footnote{If the original and the limit process do not start from the same position, we require that $\lim\limits_{n\to\infty} x_{0n}=y_0$ and $Y(0)=y_0$, as done, for example, in \cite{TSJ2014}.},
\end{equation}
where $(L_n)_{n\in\mathbb{N}}$  converge weakly to {a L\'evy process ${L_\infty}$ with characteristic triplet $(0,0,\nu)$}, i.e., $L_n\stackrel{\mathcal{L}}{\to} L_\infty$.
\end{theorem}
\begin{proof} 
The proof is reported in Section \ref{AppendixB}.
\end{proof}
The shot noise process would  converge weakly to a Gaussian limit if $L_n\stackrel{\mathcal{L}}{\to} B$, a standard Brownian motion, which would require {$\lambda_n F_{J_n}(dx)\stackrel{\mathcal{L}}{\to}0$} or 
\[
\mathbb{E}[e^{iuL_n(t)}]\to e^{-tu^2/2}.
\]
{Intuitively, this cannot happen in the {absence} of negative jumps since, e.g., $L_\infty$ is an increasing process  while the Brownian motion is not. 
This is confirmed by the following
\begin{theorem}\label{Theonew}
The limiting L\'evy-driven OU process is not of Gaussian type.
\end{theorem}
\begin{proof}
The proof is reported in Section \ref{AppendixC}.
\end{proof}}
Hence, in the presence of {nonnegative} jumps happening at high frequency with amplitudes going to zero, the nonnegative shot noise process may be replaced by a L\'evy-driven {non-Gaussian}  OU process. {The process $L_\infty$ is called background driving L\'evy process \cite{ContTankov2004} and it acts as driving noise of $Y$. Since the jumps are nonnegative, the process $L_\infty$ is also a subordinator, i.e. a L\'evy process with no diffusion component, nonnegative drift (being null) and positive increments (the jumps), whose distribution characterises both $L_\infty$ and $Y$.}

{
\subsection{Properties of the limit non-Gaussian OU processes} 
The characteristic functions of the limit $L_\infty$ and $Y$ processes have the same expression as \eqref{charL} and \eqref{mgf}, respectively, where $\lambda F_j(dx)$ is replaced by $\nu(dx)=u(x)dx$. Moreover, the mean, covariance and variance of the non-Gaussian OU limit process are given by 
\begin{eqnarray}
\label{meanOUL}\mathbb{E}[Y(t)]&=&\frac{\mathbb{E}[L_\infty(1)]}{\alpha}(1-e^{-\alpha t})+ x_0e^{-\alpha t},\\
\nonumber\label{mom3b}\textrm{Cov}(Y(t),Y(t+s))&=&\frac{\mathbb{E}[L^2_\infty(1)]}{2\alpha}e^{-\alpha s}(1-e^{-2\alpha t}),\quad {s>0,}\\
\label{varOUL}\textrm{Var}(Y(t))&=&\frac{\mathbb{E}[L^2_\infty(1)]}{2\alpha}(1-e^{-2\alpha t}),
\end{eqnarray}
where the underlying increments have mean $\mathbb{E}[L_\infty(1)]=\int_0^\infty x\nu(dx)=\int_0^\infty xu(x)dx$ and second moment $\textrm{E}[L^2_\infty(1)]=\int_0^\infty x^2u(x)dx$. 
}

The limiting L\'evy {densities of the subordinator $L_\infty$} obtained under different jump distributions  are reported in Table \ref{Table1}, {while their calculations are given in  \ref{prooflevy}.} If the jumps $J_n$ are Bernoulli or Poisson distributed, the resulting L{\'e}vy measure is that of a  PP; if the jumps are $\chi^2$ or gamma distributed, the resulting L{\'e}vy measure is that of a GP; if the jumps are inverse Gaussian distributed, the resulting L\'evy measure is that of an IGP, while beta distributed jumps yield the L{\'evy} measure of a BP. {Note that the $\chi^2$ is a special case of the gamma distribution for rate parameter $1/2$, i.e. for $\widetilde\sigma^2=2\mu$. Following \cite{Quetal2020,Quetal2021}, we talk about $Y$ being OU-Poisson, OU-Gamma, OU-inverse Gaussian (OU-IG) or OU-Beta process depending on whether the underlying subordinator $L_\infty$ is a Poisson process (PP), gamma process (GP), inverse Gaussian process (IGP) or beta process (BP) \cite{Brodericketal}, that is, $L_\infty(1)$, the process at time one, follows a Poisson, gamma, inverse Gaussian or beta distribution, respectively.} In particular, we obtain  
\begin{enumerate}
\item PP with L{\'e}vy measure $\nu(dx)=\mu \delta(x-1)dx$, where $\delta$ denotes the Dirac delta function. This is a PP with jump $1$ and intensity $\mu$. {Moreover, $\mathbb{E}[L_\infty(1)]=\mathbb{E}[L^2_\infty(1)]=\mu$.}
\item GP  with L{\'e}vy measure $\nu(dx)=\widetilde\alpha x^{-1}e^{-\beta x}dx$ concentrated on $(0,\infty)$ and independent gamma distributed increments with {shape parameter $\widetilde\alpha>0$ and rate parameter $\beta>0$, with $\mathbb{E}[L_\infty(1)]=\widetilde\alpha/\beta$ and $
\mathbb{E}[L^2_\infty(1)]=\widetilde\alpha/\beta^2$.  Looking at Table \ref{Table1}, we see that $(\widetilde\alpha,\beta)=(\mu/2,1/2)$ and $(\mu^2/\sigma^2,\mu/\sigma^2)$, depending on whether the jumps $J_n$ are $\chi^2$ or gamma distributed, respectively. Thus, in both cases $\mathbb{E}[L_\infty(1)]=\mu$, with $\mathbb{E}[L^2_\infty(1)]=\mu$ for the $\chi^2$ and $\mathbb{E}[L^2_\infty(1)]=\widetilde\sigma^2$ for the gamma case.}
\item IGP with  L{\'e}vy measure $\nu(dx)=se^{-b^2x/2}/\sqrt{2\pi x^3}dx$ concentrated on $(0,\infty)$ and inverse Gaussian distributed increments with mean $\widetilde \mu=s/b$ and shape $\widetilde\lambda=s^2$, 
{with $\mathbb{E}[L_\infty(1)]=\widetilde\mu=s/b$ and $\mathbb{E}[L^2_\infty(1)]=\widetilde\mu^3/\widetilde\lambda=s/b^3$.}
\item BP with 
L{\'e}vy measure {$\nu(dx,d\mu)=\mu\beta x^{-1}(1-x)^{\beta-1}dxd\mu$ concentrated in $(0,1)$, where $u(x)=\beta x^{-1}(1-x)^{\beta-1}$ is an improper beta function, $\mu$ is the base measure and $\beta$ is the concentration parameter, see \cite{Brodericketal} for more details on BPs and their decomposition. The underlying increments are beta distributed with $\mathbb{E}[L_\infty(1)]=\int_0^1 x \mu u(x)dx=\mu$ and $\mathbb{E}[L^2_\infty(1)]=\int_0^1 x^2\mu u(x)dx=\mu/(\beta+1)$.}
\end{enumerate}
{Note that $\mathbb{E}[L_\infty(1)]$ and $\mathbb{E}[L^2_\infty(1)]$ equal the first and second limiting infinitesimal moments reported in Table \ref{Table1}, obtained via \eqref{c1} and \eqref{c2}, respectively.}
If $J_n$ are $\chi$, generalised gamma, beta and beta prime distributed, the derived L\'evy measures  do not correspond to a known process (results not shown). 
\begin{remark}
When looking at the considered distributions meeting conditions \eqref{lambda} and \eqref{weak} and yielding a stochastic process as limit, we realise that they are all member of the exponential family. That is, $X$ is a member of a k-parameter exponential family with state space $S\subseteq[0,\infty)$ if there exists $k\in\mathbb{N}$, functions $c,\xi_1,\ldots, \xi_k:\Theta\to\mathbb{R}$, real-valued statistics $ T_1,\ldots, T_k: S\to\mathbb{R}$ and a function $h:S\to[0,\infty)$ such that the probability mass/density function of $X$  can be factorised as
\[
f_n(\theta):=h(x)c(\theta)\exp(\sum_{j=1}^k\xi_j(\theta)T_j(x)),
\]
where $c(\theta)>0$ is a normalising constant and $\theta=(\theta_1,\ldots,\theta_d), d\leq k$, where $T(X)=(T_1(X),\ldots, T_k(X))$ is a sufficient statistic for $\theta$ \cite{CasellaBerger}. 
\end{remark}
{
When the jumps $J_n$ are Bernoulli or Poisson distributed, the limit process is an OU-Poisson process, i.e. a compound Poisson process with Poissonian jumps of size 1 and intensity $\mu$ or, analogously, a shot noise process with underlying Poissonian activity with intensity $\mu$ and constant jumps with amplitude 1.} 

{Differently from the Poisson process, the gamma, IG and beta processes are not compound Poisson processes since they have non-integrable L\'evy measures, i.e. $\int_0^\infty \nu(dx)=\int_0^\infty u(x)dx=\infty$ for the GP and the IGP, and $\int_0^1\nu(dx)=\int_0^1 u(x)dx=\infty$ for the BP. These processes  are known to be infinite activity processes, as they have an infinite number of very small jumps in any finite time interval  \cite{BarndorffNielsenShephard}.  Since the integral of $u(x)$ diverges at the origin, it is of interest to integrate the L\'evy density from $x$ rather than from zero. This leads to the definition of the L\'evy-Khinchin tail function \cite{BarndorffNielsenShephard}
\[
U(x)=\int_x^\infty u(s)ds
\]
for the GP and IGP, and $U(x)=\int_x^1 u(z)dz$ for the BP. The L\'evy-Khinchin tail defines the rate at which jumps of size greater than $x$ occur. Using Maple and the parameter values from Table \ref{Table1}, the L\'evy-Khinchin tail functions for the GP, IGP and BP, denoted $U_{GP}, U_{IGP}$ and $U_{BP}$, respectively, are given by
\begin{eqnarray*}
U_{GP}(x)&=&\frac{\mu^2}{\widetilde\sigma^2}Ei_1\left(\frac{\mu x}{\widetilde\sigma}\right),\\
U_{IGP}(x)&=&s\left[\sqrt{\frac{2}{\pi x}}\exp\left(-\frac{\mu x}{2\widetilde\sigma^2}\right)
-\sqrt{\frac{\mu}{\widetilde\sigma^2}}\textrm{Erfc}\left(\sqrt{\frac{\mu x}{2\widetilde\sigma^2}}\right)\right],\\
U_{BP}(x)&=&\mu\beta \left\{(\beta-1)[{}_3F_2(1,1,2-\beta;2,2;1)+{}_3F_2(1,2,2-\beta;2,2;x)] x-\ln(x)\right\},
\end{eqnarray*}
where $Ei_1(x)$ denotes the exponential integral given by
\[
Ei_1(x)=\int_1^\infty e^{-s x}s^{-1}ds, \qquad x>0,
\]
$\textrm{Erfc}(x)$ is the complementary error function and ${}_3F_2(a_1,a_2, a_3;b_1,b_2;x)$ is the generalised hypergeometric function given by
\[
{}_3F_2(a_1,a_2, a_3;b_1,b_2;x)=\sum_{n=0}^\infty\frac{(a_1)_n(a_2)_n (a_3)_n}{(b_1)_n (b_2)_n}\frac{x^n}{n!},
\]
where $(a)_n$ is the rising factorial defined by $(a_n)=a(a+1)\cdots (a+n-1)$ for $n\in\mathbb{N}$. These L\'evy-Khinchin tails, and thus rates, increase as $x$ decreases, exploding as $x\to 0$. However, for any finite small jump amplitude, all rates are finite. }

{Finally, in analogy to the diffusion approximation of Stein's model \cite{Lansky1984}, in \ref{AppA}, we explicitly provide $\mu$ and $\sigma^2$ as a function of the rate $\lambda_n$ and the parameters of the jump size distribution. This will allow to use the limiting model as an approximation of the shot noise for large but finite rates $\lambda_n<\infty$.} 

{
\subsection{Simulation of OU-Gamma and OU-IG processes}
Differently from the OU-Poisson process,  that can be simulated as the shot noise process, the OU-Gamma, OU-IG and OU-Beta processes have underlying infinite activity L\'evy processes. The presence of an infinite number of very small jumps in any interval makes their simulation difficult.  
Common simulation schemes are based on Rosi\'nski's infinite series representation \cite{Rosinski2001} or on the numerical inversion of the underlying characteristic function \cite{GlassermanLiu2010,Chenetal2012}. These methods are not exact since they introduce truncation, discretisation or round-off errors.}

{
Here, we rely on the exact simulation algorithms proposed in \cite{Quetal2020} and \cite{Quetal2021} to generate trajectories of the OU-Gamma and OU-IG processes, respectively. The idea of these algorithms is to perform an exact distributional decomposition, i.e. the processes are decomposed into the sum of several simple elements which preserve the distribution law and are easy to sample from. In particular, the OU-Gamma process is decomposed into the sum of 
one deterministic trend, one compound Poisson random variable, modelling the finite jumps, and one gamma random variable, modelling the infinite small jumps \cite{Quetal2020}. The OU-Gamma process in \cite{Quetal2020} is written as
\begin{equation}\label{YQu}
dY(t)=-\delta Y(t)dt+ \rho dZ(t), \qquad t\geq 0,
\end{equation}
where $Z(t)$ is a gamma process with L\'evy measure $\nu(dx)=\widetilde \alpha x^{-1}e^{-\beta x}dx$, corresponding to our $L_\infty(t)$. Hence, when using the algorithm, we set $\delta=\alpha, \rho=1$ and $(\widetilde\alpha,\beta)=(\mu^2/\widetilde\sigma^2,\mu/\widetilde\sigma^2)$ (cf. Table \ref{Table1}). \\
For the OU-IG process, its transition distribution is decomposed into the sum of one compound Poisson random variable, modelling the finite jumps, and one IG distribution, modelling the infinite small jumps \cite{Quetal2021}. The OU-IG process in \cite{Quetal2021} looks like \eqref{YQu}, where $Z(t)$ is now an IG process with L\'evy density $\nu(dx)=e^{-c^2x/2}/\sqrt{2\pi x^3}dx$. Comparing the different formulations, using the limiting L\'evy density in Table \eqref{1} and the properties of the IG distribution, we have that $\delta=\alpha, c=\sqrt{\mu/\widetilde\sigma^2}$ and $\rho=s=\sqrt{\mu^3/\widetilde\sigma^2}$. 
  The algorithms, written in Matlab and kindly provided by one of the authors, have been rewritten in the computing environment R \cite{R}, using the package Rcpp 
\cite{Rcpp}, which offers a seamless integration of R and C++, drastically reducing the computational time of the algorithms. A R-package called {\em shotnoise} will be made publicly available on github %at \href{https://github.com/massimilianotamborrino/shotnoise}{https://github.com/massimilianotamborrino/shotnoise} 
upon publication, providing also the codes for running Monte Carlo simulations for the shot noise process with all considered jump distributions.}

\begin{figure}
\centering
\includegraphics[width=.65\textwidth]{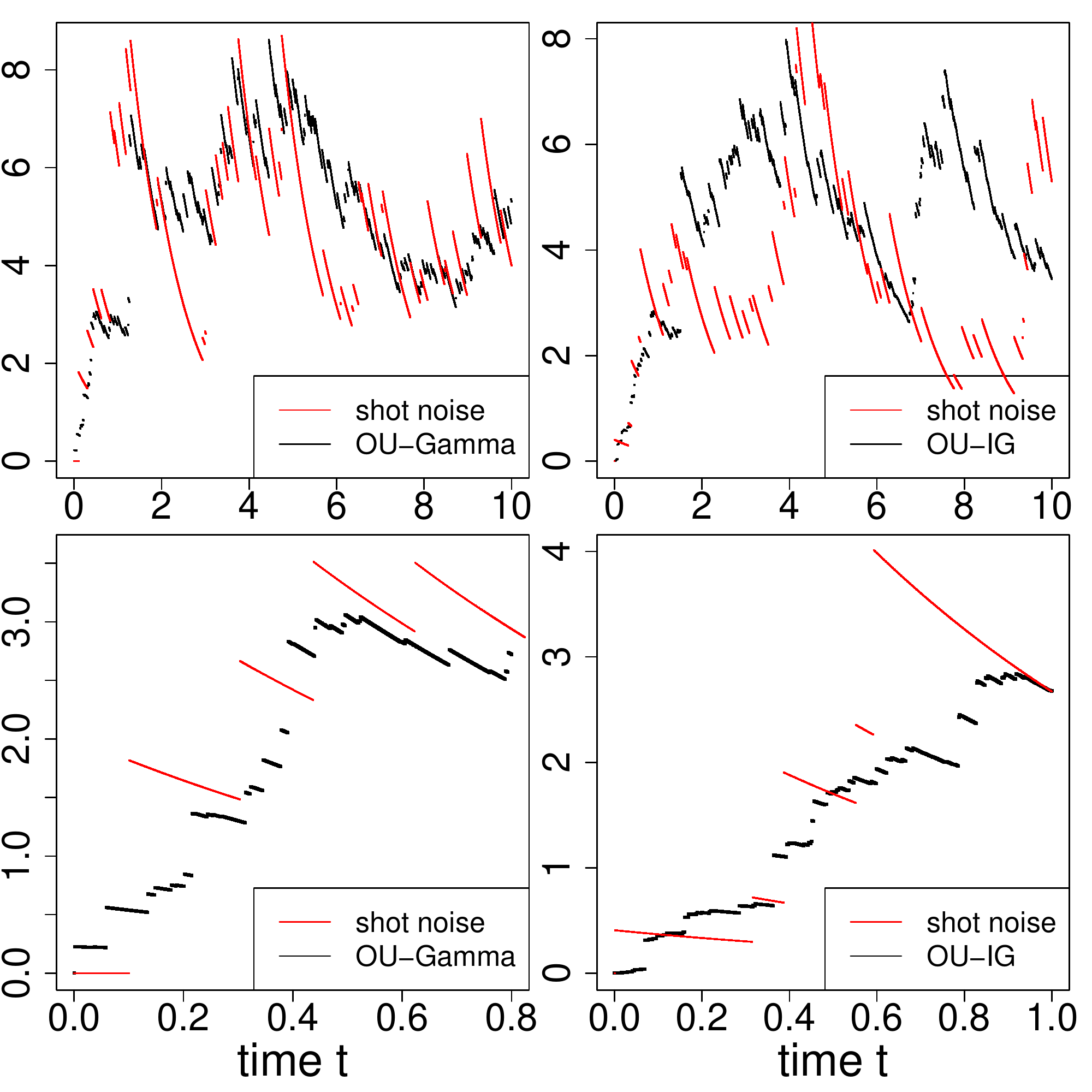}
\caption{{Simulated sample paths of L\'evy-OU processes (black lines) and shot noise (red lines) in $(0,10)$ (upper figures), with a zoom in (0,1) (bottom figures). Left panels: OU-Gamma process. Right panels: OU-IG process. Chosen parameters: $\mu=5, \widetilde\sigma^2=3, \lambda=5,\alpha=1, x_0=0$.
}}
\label{Figure0}
\end{figure}
{Sample paths of the OU-Gamma and the OU-IG processes simulated via the provided package are reported in Figure \ref{Figure0}, together with trajectories of the shot noise process with $\lambda=5$, i.e., not in asymptotic regime. The trajectories of the L\'evy-OU processes exhibit a combination of exponential decays, small jumps and large jumps, while small jumps are more rare/absent for the shot noise process. This difference becomes less visible when $\lambda$ increases, when the shot noise process is also characterised by frequent and small jumps.  
}

\section{Diffusion approximation of the shot noise process}
{Since a diffusion limit cannot be obtained for the considered sequence of nonnegative shot noise processes, as shown in Theorem \ref{Theonew}, we are now interested in deriving sufficient conditions guaranteeing this to happen. This allows us to  understand which condition is not met and, possibly, how to fulfil it by {modifying} some of the underlying assumptions. Alternatively to the convergence of the characteristic triplets,} we may consider the conditions for the weak convergence of one dimensional jump processes to diffusion processes provided by \cite{Ricciardi,GikhmanSkorokhod,KarlinTaylor}. 
 For a sequence of processes $(X_n)_{n\in\mathbb{N}}$, denote by $\Delta X_n(t)$ the increment in $(t,t+\Delta)$ {obtained by discretising the stochastic differential equation \eqref{sde} in $X_n(t)$ {with respect to time}}.  The $k$th infinitesimal moment of $X_n(t)$, denoted by $M_{k;X_n}(x), k\in\mathbb{N}$, is defined by
\begin{equation}\label{defM}
M_{k;X_n}(x):=\lim_{\Delta t\to 0^+}\frac{\mathbb{E}[(\Delta X_n(t))^k|X_n(t)=x]}{\Delta t}.
\end{equation}
An analogous definition holds for $\Delta Y(t)$ and $
M_{k;Y}(x)$, for the process $Y$.  The conditions for the weak convergence to a diffusion process are those proposed in \cite{Ricciardi,GikhmanSkorokhod,KarlinTaylor},  that we now repeat for {convenience}.

\begin{theorem}[Theorem from \cite{GikhmanSkorokhod}] A diffusion process $Y$ starting in $Y(0)=x_0$ 
 is the diffusion approximation of a sequence of jump processes $(X_n)_{n\in\mathbb{N}}$ starting in $x_{0n}=x_0$ if the following conditions are met
\begin{eqnarray}
\label{c1}&&\lim_{n\to\infty} M_{1;X_n}(x)\to M_{1;Y}(x)<\infty.\\
\label{c2}&&\lim_{n\to\infty} M_{2;X_n}(x)\to M_{2;Y}(x)>0.\\
\label{c3}&&\lim_{n\to\infty} M_{4;X_n}(x)\to M_{4;Y}(x) =  0,
\end{eqnarray}
for all $x$ in the state space of $X_n$.
\end{theorem}
It is well known that diffusion processes have all infinitesimal moments of order higher than two null \cite{KarlinTaylor}. At the same time, Pawula theorem \cite{Pawula} states that if the infinitesimal moments $M_{k}$ of a stochastic process exist for all $k\in\mathbb{N}$, the vanishing of any even order infinitesimal moment larger than two implies $M_{k} = 0$ for $k\geq 3$. This has two key consequences. First, if a process  has a finite number of nonzero infinitesimal moments, this number is {at most} two. Second, if the fourth infinitesimal moment goes to zero, then all other $M_k, k\geq 3$ goes to zero, motivating condition \eqref{c3}.

\subsection{Conditions for the diffusion approximation of the shot noise}
Throughout, we assume that $x_{0n}=x_0=y_0$. Replacing $X$ with $X_n$ in \eqref{sde}, {and discretising it {with respect to time},} we get
\begin{equation}\label{sde2}
\Delta X_n(t)=-\alpha X_n(t)\Delta t + J_n \Delta N_n(t),
\end{equation}
where $\Delta N_n(t)$ is the increment of the Poisson process in $(t,t+\Delta)$. 
The first, second and fourth infinitesimal moments of $X_n(t)$ are given by
\begin{eqnarray}
\label{m1g}M_{1;X_n}(x)&=&
-\alpha x+\lambda_n\mathbb{E}[J_n],\\
\label{m2g}M_{2;X_n}(x)&=&
\lambda_n\mathbb{E}[J_n^2],\\
\label{m4g}M_{4;X_n}(x)&=&\lambda_n\mathbb{E}[J_n^4],
\end{eqnarray}
where we used the fact that {$\mathbb{E}[(\Delta N_n(t))^k]=\lambda_n\Delta t+o((\lambda_n\Delta t)^2)$ because $\mathbb{P}(\Delta N_n(t)=1)=\lambda_n \Delta t $ and $\mathbb{P}(\Delta N_n(t)>1)=o((\lambda_n\Delta t)^2)$  
 for $k=1,2,4$, where $o(\Delta t)/\Delta t\to 0$ as $\Delta t\to 0^+$}.   The diffusion regime requires \begin{eqnarray}
\label{lim} & \lim\limits_{n\to\infty}\mathbb{E}[J_n] = 0, \qquad &\lim\limits_{n\to\infty}\lambda_n = \infty.
\end{eqnarray} 
For the shot noise process, using the infinitesimal moments of $X_n$ given by \eqref{m1g}--\eqref{m4g}, we see that the conditions \eqref{c1}-\eqref{c3} guaranteeing the diffusion approximation  are given by
\begin{eqnarray}
\label{EJn}&& \lim_{n\to\infty}\lambda_n \mathbb{E}[J_n]= \mu\in[0,\infty),\\
\label{EJ2n}&&\lim_{n\to \infty}  \lambda_n \mathbb{E}[J^2_n]= \sigma^2, \quad \sigma>0, \\
\label{EJ4n}&&\lim_{n\to \infty} \lambda_n \mathbb{E}[J^4_n]= 0,
\end{eqnarray}
which are equivalent to 
 \begin{equation}\label{condJn}
\mathbb{E}[J_n]=O(\lambda^{-1}_n), \qquad  \mathbb{E}[J^2_n]={\Theta}(\lambda^{-1}_n), \qquad \mathbb{E}[J^4_n]={o}(\lambda^{-1}_n),
\end{equation}
where {for $a_n,b_n>0, a_n=O(b_n), a_n=\Theta(b_n)$ or $a_n=o(b_n)$ if $a_n/b_n$ converges to $l\in [0,\infty), l\in(0,\infty)$ or $0$, respectively, as $n\to\infty$.} That is, the mean and the second moment of the jump amplitudes $J_n$ should go to zero {at the same rate as} $\lambda_n$ goes to infinity, while the fourth moment should go to zero faster than $1/\lambda_n$. 
\begin{remark}
As required in \eqref{EJ2n}, the limit of the second infinitesimal moment should not be zero, otherwise the limit process will be deterministic. 
\end{remark} 
\begin{remark}
Conditions \eqref{EJn}, \eqref{EJ2n}, \eqref{EJ4n} under \eqref{lim} are the same as (3.9), (3.11), (3.18) under (2.10) for the diffusion approximation of a jump process with synaptic reversal potential in \cite{LanskyLanska1987}.
\end{remark}
\begin{remark}\label{Remark4}{
When $X$ is a member of a $k$-parameter exponential family and $T_j(X)=X$ for some of  the $j=1,\ldots, k$ statistics, conditions \eqref{EJn}, \eqref{EJ2n},  \eqref{EJ4n}become
\begin{eqnarray*}
&&-\lim_{n\to\infty}\lambda_n\frac{\partial_j\log c(\theta)}{\partial_j\xi(\theta)}=\mu,\\
&& 
\lim_{n\to\infty}\frac{\lambda_n}{\partial^2_{j}\xi_j(\theta)}[-\partial_j^2\log c(\theta)+(\partial_j\log c(\theta))^2]=\sigma^2,\\
&&\lim_{n\to\infty}\frac{\lambda_n\left\{-\partial_j^4\log c(\theta)+\partial_j^3 \log c(\theta)\partial_j \log c(\theta)[1-3\partial_j^2 \log c(\theta)]+3
\partial_j^2 \log c(\theta)[\partial_j^2 \log c(\theta)-2(\partial_j \log c(\theta))^2]\right\}}
{\partial^4_{j}\xi_j(\theta)}=0,
\end{eqnarray*}
where $\partial_j$ denotes the partial derivative with respect to $j$, and  moments have been derived from the moment generating function \cite{CasellaBerger}.}
\end{remark}
\begin{remark}
If the conditions are fulfilled, the diffusion limit of the sequence of shot noise processes is  {a Gaussian} OU  process starting in $x_0=x_{0n}=y_0$ with mean, covariance and variance given by 
\begin{eqnarray}
\label{meanOU}\mathbb{E}[Y(t)]&=&\frac{\mu}{\alpha}(1-e^{-\alpha t})+ y_0e^{-\alpha t},\\
\label{covOU}\textrm{Cov}(Y(t),Y(t+s))&=&\frac{\sigma^2}{2\alpha}e^{-\alpha s}(1-e^{-2\alpha t}), \\
\label{varOU} \textrm{Var}(Y(t))&=&\frac{\sigma^2}{2\alpha}(1-e^{-2\alpha t}).
\end{eqnarray}
{This diffusion process, if existing, would coincide with that obtained by the usual and the Gaussian approximations.}
\end{remark}

\begin{conjecture}
A sequence of nonnegative random variables $(J_n)_{n\in\mathbb{N}}$  satisfying the conditions  \eqref{EJn}-\eqref{EJ4n} under \eqref{lim} does not exist. 
\end{conjecture}
\noindent Hence, as also shown in Theorem \ref{Theonew}, unless the jump amplitudes do not depend on $n$ \cite{DassiosJang2005}, the OU  process, {obtained via the usual and Gaussian approximations,} 
 cannot be obtained as diffusion approximation of the shot noise process. In particular, when \eqref{lim}, \eqref{EJn}, \eqref{EJ2n} are satisfied, \eqref{EJ4n} is not fulfilled, i.e. the fourth infinitesimal moment does not vanish, unless $J_n$ assumes negative values, as shown in the following.

\subsection{Examples}
We now consider different families of jump distributions, discuss the consequences of violating the vanishing of the fourth infinitesimal moment and investigate the errors when replacing the shot noise with the Gaussian OU process.  

\subsubsection{{Degenerate (constant)} and exponential distributions}
Conditions \eqref{EJn} and \eqref{EJ2n} cannot be violated, otherwise the limit process  would either have a mean going to infinity or be a deterministic process. {This is what happens when the jump size is constant (deterministic), i.e. $J_n=j_n>0$, and thus all moments are equal. Indeed, if $\lambda_n j^2_n\to \sigma^2$, then $\lambda_nj_n\to \infty$, i.e. the mean of the limit process goes to infinity. On the contrary, if $\lambda_n j_n\to \mu$, then $\lambda_nj_n^2\to 0$, i.e. the limit process is deterministic. This is also what } happens if $J_n$ is exponentially distributed with mean $\theta_n$. Indeed, if $\lambda_n\mathbb{E}[J_n]=\lambda_n\theta_n \to \mu$, then $\lambda_n \mathbb{E}[J^2_n]=2\lambda_n\theta^2_n\to 0$ as $n\to\infty$.  Hence, a shot noise process with {constant} or exponential distributed jumps yields a deterministic and not a diffusion process, {making thus a Gaussian approximation not suitable/accurate}. 
This may explain why the considered diffusion approximation {for the exponential case} misrepresents the subthreshold voltage distribution for certain types of synaptic drive, see, e.g. \cite{RichardsonGerstnerChaos}, or the lack of fit of the derived firing statistics \cite{DL,RichardsonSwarbrick2010} compared to the alternative approaches developed there. Similar results hold when considering jumps to be lognormal distributed (results not shown).

\subsubsection{Results when the fourth infinitesimal moment does not vanish}\label{section423}

\begin{table}[t]
%\hspace{-.8cm}
\begin{tabular}{|c|c|c|c|c|}
\hline
Distribution & Parameters &
 $\lambda_n\mathbb{E}[J^2_n]{\to \sigma^2}$&$%\lim\limits_{n\to\infty} 
 \lambda_n\mathbb{E}[J^4_n]{\to} M_{4;Y}(x)$& ${\lambda_n{f_{J_n}(x)}\stackrel{\mathcal{L}}{\to}{u(x)}}$\\ \hline
Bernoulli &$ p=\frac{\mu}{\lambda_n} $& $ \mu  $& $ \mu $ &$\mu{\delta(x-1)}$ (PP)\\ \hline
Poisson, {Poi$(\widetilde\lambda)$}&${\widetilde\lambda}= \frac{\mu}{\lambda_n}$& $ \mu $& $ \mu $&$\mu{\delta(x-1)}$\ (PP)\\ \hline 
%$\chi^2(df)$ &${df=} \frac{\mu}{\lambda_n}$ &$2\mu$&$48\mu$&$\frac{\mu}{2}x^{-1}e^{-x/2}$ (GP)\\ \hline
$\Gamma $& ${\widetilde\alpha}= \frac{\mu}{\lambda_n}{\beta}$&&&\\
{shape $\widetilde\alpha$, rate $\beta$}
&$ {\beta}={\frac{\mu}{\widetilde\sigma^2}}$& ${\widetilde\sigma^2}$&$6\frac{{\widetilde\sigma^6}}{\mu^2}$ &$\frac{\mu^2}{{\widetilde\sigma^2}}x^{-1}e^{-\mu x/{\widetilde\sigma^2}}$ (GP) \\ \hline
inverse Gaussian & ${\widetilde\mu =} \frac{\mu}{\lambda_n}$&&&$
\frac{s}{\sqrt{2\pi x^3}}e^{-\mu x/(2{\widetilde\sigma^2})}$ (IGP)\\
IG({mean} {$\widetilde\mu$}, {shape}  {$\widetilde\lambda$})& {$\widetilde\lambda$} = $\frac{\mu^3}{\lambda_n^2{\widetilde\sigma^2}}$ & ${\widetilde\sigma^2}$ & $\frac{15{\widetilde\sigma^6}}{\mu^2}$&$s=\sqrt{\frac{\mu^3}{{\widetilde\sigma^2}}}$\\ \hline
beta &  $ {\widetilde\alpha} =\frac{\mu\beta}{\lambda_n}$&&&\\
 (shape ${\widetilde\alpha}$, shape $\beta$) &  $\beta$ & $\frac{\mu}{\beta+1}$& $\frac{6\mu}{(\beta+1)(\beta+2)(\beta+3)}$&$\mu\beta x^{-1}(1-x)^{\beta-1}$ (BP)\\ \hline
\end{tabular}
\caption{List of the nonnegative jump amplitude distributions yielding non-null first two infinitesimal moments, i.e. satisfying conditions \eqref{EJn},  \eqref{EJ2n} under \eqref{lim}. For all jump distributions, $\lim\limits_{n\to \infty}\lambda_n\mathbb{E}[J_n]=\mu$. 
 Here {$u(x)$} denotes the L{\'e}vy {density} of the limiting L\'evy-driven OU process, cf. Section \ref{Section4}. Other distributions fulfilling \eqref{EJn} and \eqref{EJ2n} are the $\chi$, the generalised gamma and the beta prime distributions (results not shown). {The $\chi^2$ distribution can be obtained as a special case of the gamma with rate $\beta=1/2$, i.e. $\widetilde\sigma^2=2\mu$. 
 All pmfs and pdfs are reported in \ref{prooflevy}, together with the calculations to obtain the limit L\'evy densities $u(x)$. Here PP, GP, IGP and BP denote the Poisson process, the gamma process,  the inverse Gaussian process and the beta process, respectively, defined in Section 3. 
}}
\label{Table1} 
\end{table}

In Table \ref{Table1}, we report a list of discrete and continuous nonnegative random variables $J_n$ satisfying  conditions \eqref{EJn} and \eqref{EJ2n} under \eqref{lim}. Other random variables fulfilling these requirements are the $\chi$, the generalised gamma and the beta prime distribution (results not shown). All of them are member of the exponential family, with the Bernoulli, Poisson, gamma and inverse Gaussian distributions having {sufficient statistic} $T(X)=X$ (the first {two}) or $T_1(X)=X$ (the {others}), meaning that the conditions of Remark \ref{Remark4} could be alternatively verified. 
We provide both the distribution parameters and the resulting first, second and fourth infinitesimal moments. Since none of the fourth infinitesimal moment vanishes, condition \eqref{EJ4n} is not fulfilled, meaning that the limit process is not a diffusion. However, the fourth infinitesimal moment can be made arbitrarily small by letting  
$\mu\to\infty$ if the jumps are gamma, inverse Gaussian or beta distributed, {and $\widetilde\sigma^2$ does not depend on $\mu$}. 
 Intuitively, if $\mu$ increases, the probability that the OU process at time $t$ assumes negative values decreases, reducing thus the discrepancy between the state space of the original and the limit process, improving thus the quality of the approximation. For the other considered jump distributions, the fourth infinitesimal moment cannot vanish, otherwise all infinitesimal moments would vanish, yielding a degenerate limit in a point. Finally, if $J_n$ is beta distributed, one of its underlying parameter can be arbitrarily chosen in a way such that,  for fixed $\mu$ and $\sigma^2$, it yields the smallest fourth infinitesimal moment among the  considered distributions. 

\subsubsection{Error caused by replacing the shot noise with the {L\'evy OU} and {Gaussian OU} processes}

{To compare the quality of the approximation of the shot noise process with the non-Gaussian OU or the OU processes, and to investigate the role played by the jump amplitude distributions on that, {we perform $10^6$ Monte Carlo simulations of the processes of interest, via the {\em shotnoise} R-package released  on github upon publication.}} 
\begin{figure}
\begin{center}
\includegraphics[width=.8\textwidth]{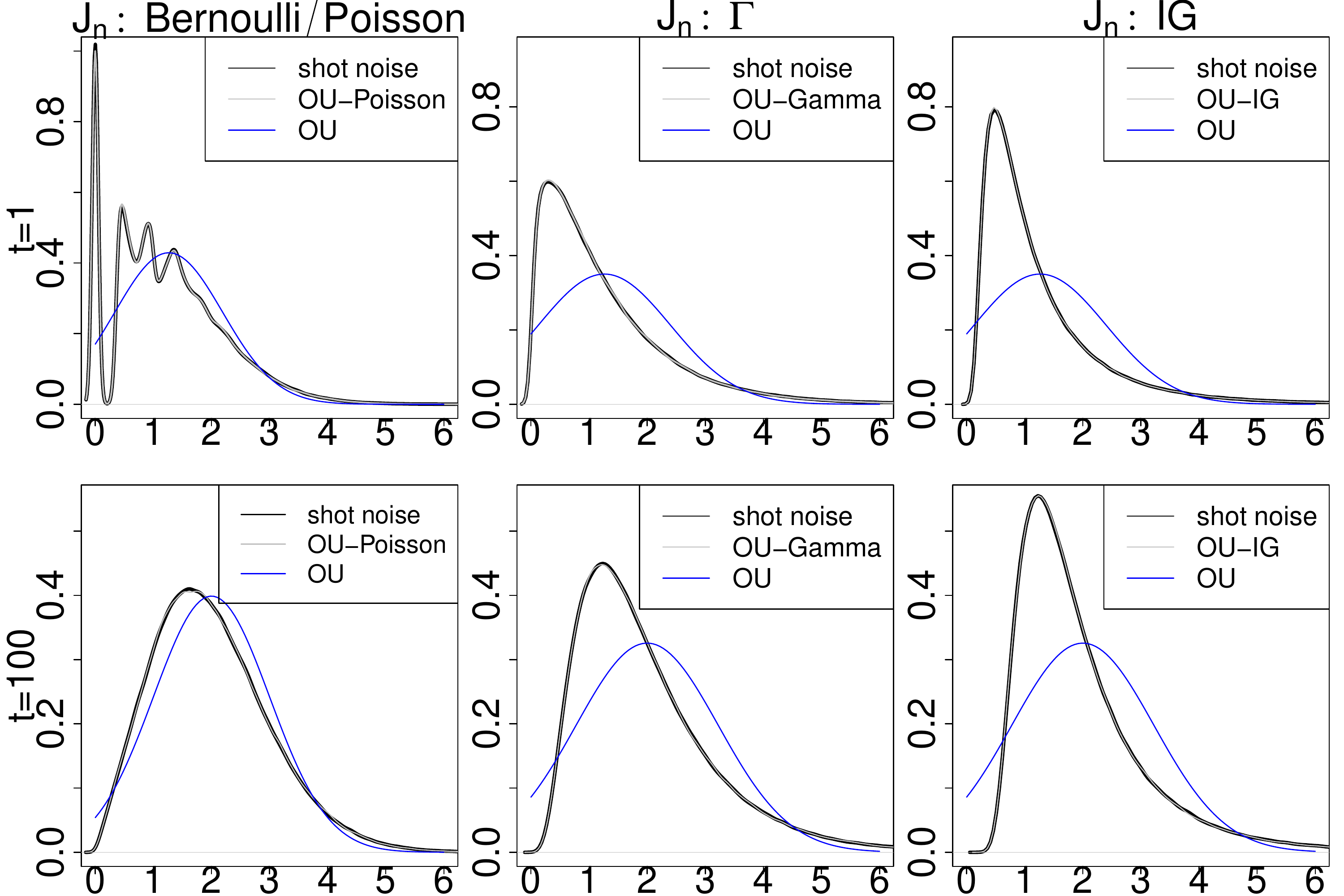}
\end{center}
\caption{Density of {the shot noise process $X_n(t)$ (black lines), the non-Gaussian OU process \ref{Levy} (grey lines) and the OU process (blue lines) obtained through $10^6$ Monte Carlo simulations} 
if $J_n$ is Bernoulli {or Poisson (left panels)}, gamma (central panels) or inverse Gaussian distributed (right panels), with parameters given in Table \ref{Table1} for t={1 (upper figures) and $t=100$ (bottom figures)}, $\mu={1}, {\widetilde\sigma^2}=3, \lambda= 1000, \alpha=1, x_0=0$.}
\label{FigureMC}
\end{figure}

{In Figure \ref{FigureMC}, we report 
the pdfs of the shot noise process $X_n(t)$, the limit L\'evy driven OU processes and the Gaussian OU process obtained via simulations at time $t=1$ (top panels) and $t=100$ (bottom panels) for different underlying jump distributions. The L\'evy-driven OU  processes, now denoted $Y_L(t)$, are solutions of \eqref{Levy}, with mean \eqref{meanOUL} and variance \eqref{varOUL},
with $\mathbb{E}[L_\infty(1)]=\mu$ and $\mathbb{E}[L_\infty^2(1)]=\sigma^2$, with $\sigma^2$ and L\'evy densities given in Table \ref{Table1}. The Gaussian OU process, now denoted $Y_G(t)$, is normally distributed with mean and variance given by \eqref{meanOU} and \eqref{varOU}, respectively. The considered times allow to compare the performance of the approximations at the beginning of the evolution and in the asymptotic/stationary regime.
 In all cases, the non-Gaussian OU processes successfully approximate the shot noise process,
 with overlapping densities, while the OU process yields a poor approximation, especially for $t=1$. }
 
{To measure the error when approximating the shot noise process $X_n(t)$ with {the L\'evy-driven OU $Y_L(t)$ or the Gaussian OU $Y_G(t)$, and to study its dependence on the time $t$ and on the underlying parameters}, we consider the integrated absolute error (IAE) defined as
\[
\textrm{IAE}(f_{X_n(t)},f_{Y{_a}(t)}):=\int_0^\infty|f_{X_n(t)}(x)-f_{Y{_a}(t)}(x)|dx,
\]
{where $a$ is either $L$ or $G$.} 
In Figure \ref{Fig2}, we report this error {for the Gaussian OU (top panels) and the L\'evy-driven OU (bottom panels)} as a function of $t$ (left and central figures) and $\mu$ (right figures)  for different jump distributions  and $x_{0n}=x_0=y_0={0}$. {The derived L\'evy-driven OU processes $Y_L$ yield the best approximation of the shot noise, outperforming the OU $Y_G$ in all considered scenarios, with IAEs at least 15 times smaller, except for the OU-Poisson when $\mu=3$. In that case, the IAE has a non-monotonic behavior in $t$, probably due to the underlying discrete Poissonian jump nature. Except for this, all considered jump distributions yield similar IAEs for $Y_ L(t)$, which are below $1\%$ and approximately constant in $t$ and $\mu$ for large values of $\lambda$, unless $\mu$ is very small. }   
The quality of the approximation {of the Gaussian OU} improves if either $\mu$ or $t$ increase, while it decreases if ${\widetilde\sigma^2}$ increases and $J_n$ depends on it (figure not shown). This can be explained as follows. 
While the shot noise is always nonnegative, the probability of the OU of being nonnegative, knowing that its state space is $\mathbb{R}$,  is increasing in $t$ and $\mu$ and decreasing in $\sigma^2$, being given by 
\[
\mathbb{P}(Y{_G}(t)\geq 0)=1-\Phi\left(-\frac{\mu(1-e^{-\alpha t})}{\sqrt{\sigma^2\alpha(1-e^{-2\alpha t})}}\right),
\]
where $\Phi(x)$ denotes the cdf of a standard normal distribution. {For the Gaussian OU,}  among the considered jump distributions  of $J_n$, the Bernoulli and the Poisson, {having the same fourth infinitesimal moments (cf. Table \ref{1}}),  yield similar IAEs. These errors are smaller than those from the other distributions unless $t$ is small and $\mu$ is large (cf. Figure \ref{Fig2}, {top panels}), {in which case the gamma} distribution yields the lowest IAE. {Hence, choosing a jump distribution yielding a fourth infinitesimal moment lower than another, does not necessarily guarantee a smaller IAE {of the Gaussian OU}, as it can be observed by comparing, for example, the results from the inverse Gaussian and the Bernoulli distributions (cf. Table \ref{Table1}). %Finally, note that, despite the yielded fourth infinitesimal moment is the same, the IAEs of the Bernoulli and the Poisson distributions may differ, see for example Figure \ref{Fig2}, right panel for $t=2$ and $\mu\in(0,4)$. 

\begin{figure}[h!]
\includegraphics[width=1.03\textwidth]{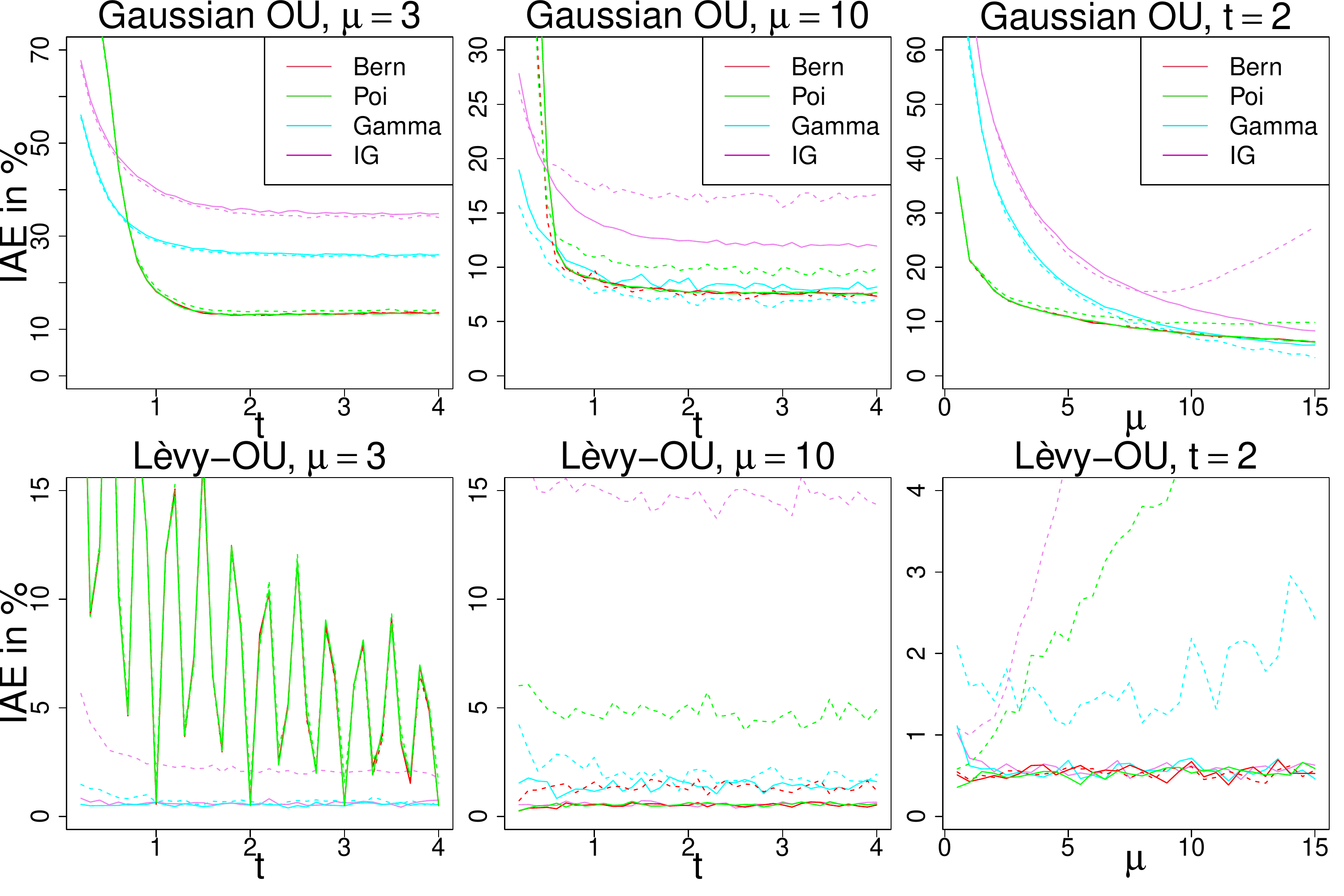}
\caption{{IAE of the shot noise vs the Gaussian OU process (top panels), } IAE($f_{X_n(t)},f_{Y{_G}(t)}$) {and IAE of the shot noise vs the L\'evy-driven OU processes (bottom panels), IAE($f_{X_n(t)},f_{Y_L(t)}$),} 
as a function of $t$ (left {and middle  figures}) and $\mu$ (right figures) when $\alpha=1, {\widetilde\sigma^2}= 3, {x_0=y_0=0}$ and the jump amplitude $J_n$ is Bernoulli (red lines), Poisson (green lines), gamma (light blue lines) or IG distributed with parameters as in Table \ref{Table1} and $\lambda=100$ (dashed lines) or ${10000}$ (solid lines). 
Left {figures: $\mu=3$. Central figures: $\mu=10$.} Right figures: $t=2$.}
\label{Fig2}
\end{figure}

\subsection{Results for negative jump amplitudes}\label{Section4.2}
Throughout this section, we relax the assumption of having nonnegative jumps, allowing $J_n$ to assume negative values but with probability going to 0 as $n\to \infty$. Let $J_n$ be defined by
 \begin{equation}\label{Jnew}
 J_n=\begin{cases}
J_n^+ & \textrm{ with } \mathbb{P}(J_n=J_n^+)= 1-f_n;\\
J_n^- & \textrm{ with } \mathbb{P}(J_n=J_n^-)= f_n,
\end{cases}
\end{equation}
where $J_n^+, J_n^-$ {are nonnegative} and nonpositive random variables, respectively, and $f_n$ is a {real number} in $[0,1)$ such that $(f_n)_{n\in\mathbb{N}}$, 
$f_n\to 0$ as $n\to \infty$.  
Nonnegative jump distributions can be immediately recovered by setting $f_n=0$. If $f_n\in(0,1)$, negative jumps happen with probability going to 0 as $n\to\infty$. {Nevertheless,} this is enough to guarantee that the $J_n$ defined by \eqref{Jnew} fulfils conditions \eqref{EJn}-\eqref{EJ4n}, yielding thus the OU process with mean \eqref{meanOU} and variance \eqref{varOU} as the limit process of $X_n(t)$. We prove this result for {generic $J_n^+$ and $f_n$, assuming $J^-_n$ to follow a univariate degenerate distribution, but} the theorem and the limit OU process hold for other suitable choices of $J_n$ and $f_n$.

\begin{theorem}\label{Theo3}
Under assumption \eqref{lim}, consider $f_n=\lambda_n^{-1+c_1}, c_1\in(0,2/3)$ such that $f_n\to0$ as $n\to\infty$ and 
\begin{equation}\label{newc}
\lim_{n\to\infty}\lambda_n f_n=\infty.
\end{equation}
Let $J_n^-$ be a univariate degenerate distribution assuming only the value $\theta_n=(-\lambda_n^{-c_1+c_2}+\lambda_n^{-c_1}\mu)<0$ {with $c_2\in (0,c_1/2)$},  with pdf, cdf, $k$-moment, $k\in\mathbb{N}$ and variance given by
\[
f_{J_n^-}(x)=\delta(x-\theta_n), \qquad F_{J_n^-}(x)=\mathbbm{1}_{\{x\geq \theta_n\}}, \qquad \mathbb{E}[(J_n^-)^k]=\theta_n^k \qquad \textrm{Var}(J_n^-)=0,
\]
where $\mathbbm{1}_A$  denotes the indicator function of the set $A$. Let $J_n^+$ be a nonnegative distribution satisfying
\begin{equation}\label{limnew}
\mathbb{E}[J^+_n]=\lambda_n^{-1+c_2}, \quad\mathbb{E}[(J^+_n)^2]=\sigma^2\lambda_n^{-1}, \qquad \mathbb{E}[(J^+_n)^4]=\lambda_n^{-1-c_3}, 
\end{equation}
with $c_3>0$. Then, conditions \eqref{EJn}, \eqref{EJ2n} and \eqref{EJ4n} are fulfilled.
\end{theorem}

\begin{proof}
The proof is reported in Section \ref{AppendixA}. 
\end{proof}
{Note that the negative amplitude $\theta_n$ goes to zero as $n\to\infty$, since $\lambda_n\to\infty$,  $-c_1<0$ and $-c_1+c_2<0$ being $c_2< c_1/2$.} 
\begin{corollary}\label{cor2}
A nonnegative distribution $J^+_n$ satisfying the assumptions of Theorem \ref{Theo3} exists.
\end{corollary}
\begin{proof}
The proof is reported in Section \ref{AppendixD}.
\end{proof}
\begin{corollary}
The assumptions of Theorem \ref{Theo3} can be generalised to
\begin{equation}\label{newcor}
\mathbb{E}[J^+_n]=\lambda_n^{-1+c_2}+O(\lambda_n^{-1}), \qquad \mathbb{E}[(J^+_n)^2]=\Theta(\lambda_n^{-1}), \qquad \mathbb{E}[(J^+_n)^4]=o(\lambda_n^{-1}),
\end{equation}
with a   degenerate univariate negative distribution given by $\theta_n={-f_n^{-1}}\mathbb{E}[J^+_n]+O(\lambda_n^{-c_1})$, {with $f_n=\lambda_n^{-1+c_1}$.}
\end{corollary}
\begin{proof}
The result follows mimicking the proof of Theorem \ref{Theo3} and noting that $\lambda_n^{-c_1+c_2}$ in $\theta_n$ is in fact $-{f_n^{-1}}\mathbb{E}[J^+_n]$. 
\end{proof} 
{
\begin{remark}
While the conditions on the second and fourth moment of $J^+_n$ in \eqref{newcor} are the same as those in \eqref{condJn}, this is not the case for the first moment. While in the absence of negative jumps, the mean of the positive jumps has to go to zero as the same rate as $\lambda_n$ to meet \eqref{EJn}, in the presence of negative jumps, $\mathbb{E}[J^+_n]$ goes to zero slower than $\lambda_n$, yielding  $\lim\limits_{n\to\infty}\lambda_n\mathbb{E}[J^+_n]=\lim\limits_{n\to\infty}\lambda_n^{c_2}=\infty$. However, even if $\lambda_n\mathbb{E}[J^+_n]$ diverges as $n\to\infty$, condition \eqref{EJn} is fulfilled thanks to the presence of negative jumps, as shown in Theorem \ref{Theo3}. Examples of nonnegative distributions fulfilling \eqref{limnew} and \eqref{newcor} but not \eqref{condJn} are given in the proof of Corollary \ref{cor2} in Section \ref{AppendixD}, and include the gamma and inverse gaussian distributions with $\mu=\lambda_n^{c_2}$
and $\widetilde\sigma^2$ not depending on $\mu$ (or being constant as $\mu\to\infty$), and the beta distribution with $\mu=\lambda_n^{c_2}$ and $\beta=\mu/\widetilde\sigma^2=\lambda_n^{c_2}/\widetilde\sigma^2$ and parameters given in Table \ref{Table1}. Thus, choosing any of these nonnegative distributions and nonnegative jumps $J^-_n$ distributed as in Theorem \ref{Theo3} allows to obtain a diffusion limit. 
\end{remark}}
The result of Theorem \ref{Theo3} can be explained as follows.  Since $J_n$ can assume negative values, we could rewrite the stochastic differential equation \eqref{sde} as 
\begin{equation}\label{sde3}
dX_n(t)=-\alpha X_n(t)dt + E_n dN_n^+(t)+I_n dN_n^{-}(t),
\end{equation}
where $E_n=J_n^+$, $I_n=J^-_n=\theta_n<0$ and $N_n^+(t)$  and $N_n^-(t)$ are Poisson processes with intensities $\lambda_n (1-f_n)$ and $\lambda_n f_n$, respectively. {This process corresponds to a shot noise process with real valued shot effects, known as generalised Stein's model \cite{Stein1965}, with random jump amplitudes $E_n$ and $I_n$ modelling the excitatory and the inhibitory components, respectively, and intensities
\[
\lambda_n (1-f_n)\to  \infty \qquad \lambda_n f_n \to \infty, \quad \textrm{ as } n\to \infty.
\]  
The crucial assumption enabling the convergence to the Gaussian OU process is \eqref{newc}. Indeed, despite the probability of having negative jumps vanishes since $f_n\to 0$ as $n\to\infty$, {it guarantees that} the frequency of the negative jumps explodes,  $\lambda_n f_n\to \infty$, enabling thus the diffusion limit. {The weak convergence of a {sequence of Stein processes with constant jump amplitudes $\to$ sequence of Stein processes (i.e. generalised Stein processes with constant jump amplitudes $E_n$ and $I_n$)} to an OU process have been already shown in \cite{Lansky1984}, and it can be seen as a direct consequence of {the central limit theorem} for a shot effect assuming real values \cite{Rice1977}.}

\section{Discussion}
{Poissonian shot noise processes are used to model single neuronal membrane potential, synaptic input currents impinging on the neuron and conductances  in single neuron modelling. When the positive jumps (excitatory inputs) impinge on the neuron with frequency going to infinity and jump amplitudes going to zero, diffusion approximations have been considered to approximate and/or compare them  with Gaussian OU processes \cite{DL,RichardsonGerstnerChaos,MelansonLongtin2019,RichardsonSwarbrick2010}, even when the jump amplitudes are exponential distributed.  

{If a sequence of jump processes converges weakly to a diffusion process, the notions of weak convergence, diffusion approximation and usual approximation yield the same limit diffusion process, which coincides with that from the Gaussian approximation if the diffusion process is Gaussian. However, if the limit process obtained via the weak convergence approach is not continuous, as it happens here for the {nonnegative} shot noise, only the notion of diffusion approximation detects this via a non-vanishing fourth infinitesimal moment, yielding a warning in the choice of approximating the initial process by a diffusion.} While the conditions guaranteeing the weak convergence of stochastic process may be too impractical or technical to be investigated, when dealing with one-dimensional processes, checking the convergence of the first two infinitesimal moments and the vanishing of the fourth {infinitesimal moment (as proposed in \cite{Ricciardi})}, represents an {equivalent (thanks to the results in \cite{GikhmanSkorokhod})}, intuitive and powerful tool which could be  adopted when performing diffusion approximations, {improving the reliability and quality of the approximation.}}

{Here, we prove that the {nonnegative} shot noise converges weakly to a L\'evy-driven {non-Gaussian} OU process {and we characterised the limiting L\'evy measure based on the underlying jump distributions}, {explicitly deriving mean, variance and characteristic function of the limit process as well as providing a R-package for its exact simulation. The derived non-Gaussian OU processes outperform the Gaussian OU in terms of approximation of the shot noise.
}  From a  modelling point of view, the derived process could then be used to replace the shot noise process modelling membrane voltages, synaptic input currents or conductances, improving the existing results on single-neuron modelling and their firing statistics  {based on the Gaussian OU process as a result of the usual and Gaussian approximations.}  A qualitative study of the introduced improvement needs to be carried {out}. 

{The lack of a limiting diffusion process is not specific for the  shot noise process. Similar results hold for all models involving only nonnegative and/or nonpositive random variables for the jumps, e.g.  neuronal models with synaptic reversal potentials, see e.g. \cite{Cupera2014}. For example, the conditions guaranteeing the diffusion approximation of the jump model in \cite{LanskyLanska1987} (cf. Theorem 1) are not met, meaning that the provided diffusion {process cannot be obtained via a diffusion approximation, but only as a usual approximation.}

Several generalisations of \eqref{1} have been proposed in the literature, e.g. non-Poisson inputs,  non-renewal dynamics, non-stationary dynamics, general shot effect $g(t)$ instead of the considered $g(t)=J^{[k]}e^{-\alpha t}$,  all under the name of generalised shot noise process \cite{Rice1977}, or the recently-proposed random process with immigration \cite{Iksanovetal2017}. All are characterised by being piecewise-deterministic Markov processes, also known as stochastic hybrid systems{, i.e. processes with deterministic behavior between jumps}. 
 Depending on the underlying conditions for the generalised shot noise processes (e.g. the shot effects, and thus the {shot noises themselves}, are commonly assumed to  assume values in $\mathbb{R}$, see e.g.  \cite{PangZhou2018,Rice1977,KluppelbergKuhn2004}) and on how they are drifted, rescaled in time and normalized, functional central limit theorems have been proved, which have Gaussian processes \cite{Rice1977,Iksanovetal2018,Papoulis1971}, self-similar Gaussian processes \cite{KluppelbergMikosch1995}, infinite-variance stable processes \cite{Kluppelbergetal2003}, fractional Brownian motion \cite{KluppelbergKuhn2004}, stable (non-Gaussian) processes (cf. \cite{Iksanovetal2017} and references therein), stationary \cite{Iksanovetal2017b} or non-stationary \cite{PangZhou2018} stochastic processes as limits. Our result contributes to this discussion, under the assumption of a nonnegative and non drifted shot noise process.} 

\section{Proofs}
\subsection{Proof of Theorem \ref{Theo1}}\label{AppendixB}
{To prove $L_n\stackrel{\mathcal{L}}{\to} L_\infty$ we rely on the convergence of characteristic triplets, as suggested in \cite{Jacod}. By looking at the characteristic function of $L_n$ in  \eqref{charL}, we see that the characteristic triplet of $L_n$ is given by $(0,0,\lambda_nF_{J_n})$, where the first, second and third entries represent the drift (here null), the diffusion component (here null),  and the L\'evy measure, here $\lambda_nF_{J_n}$, respectively. Under conditions \eqref{lambda} and \eqref{weak}, 
the characteristic triplet converges weakly to $(0,0,\nu)$ or, analogously,} the characteristic function of $L_n$ computed from \eqref{charL}  converges  weakly to 
\[
\mathbb{E}[e^{iuL_\infty(t)}]=\exp\left(t\int_{{0}}^{{\infty}}(e^{iux}-1)\nu(dx)\right).
\]
Hence, $L_\infty$ is a L{\'evy} process with initial value 0, drift 0, no Gaussian part and L\'evy measure $\nu$ \eqref{weak}.} 
{From its definition, $X_n$ belongs to $\mathcal{D}(\mathbb{R^+})$, which is a Polish space with the Skorohod topology \cite{BilConv}.} Since $X_n$ is a continuous functional of $L_n$, see \eqref{sde}, {and $x_{0n}=x_0=y_0$ (for hypothesis)}, the weak convergence of $L_n$ implies the weak convergence of $X_n$ {from} the continuous mapping theorem. In particular, 
\begin{eqnarray*}
\mathbb{E}[e^{iuX_n(t)}]&=&\exp\left(\int_{{0}}^{{\infty}}\int_0^t (e^{iue^{-\alpha y}x}-1)dy \lambda_nF_{J_n}(dx)+iux_0e^{-\alpha t}\right)\\
&\to&
\exp\left(\int_{{0}}^{{\infty}}\int_0^t (e^{iue^{-\alpha y}x}-1)dy \nu(dx)+iux_0e^{-\alpha t}\right)
\end{eqnarray*}
{guarantees} that the limit process of $X_n$ is $Y$ given by \eqref{Levy}, as $n\to\infty$.

\subsection{Proof of Theorem \ref{Theonew}}\label{AppendixC}
{Under the assumption of nonnegative jumps, for each $n$, the trajectories of $L_n$ are non-decreasing. In particular if $A$ is the following closed set in the Skorohod topology $\mathcal{D}(\mathbb{R^+})$,
\[
A := \{w\in\mathcal{D}(\mathbb{R}^+), \inf_{s\in[0,1]} w(s) = w(0)\},
\] 
then the law of $L_n$ is such that $\mathbb{P}(L_n\in A)=1$. When considering a Brownian motion {$B=\{B_t, t\geq 0\}$}, or more generally any regular diffusion, we have $\mathbb{P}(B\in A)=0$. Hence, $\lim\sup_n\mathbb{P}(L_n\in A)=1>\mathbb{P}(B\in A)$, i.e., statement 3 of the  Portmanteau theorem \cite{BilConv} is not met, implying that the equivalent result $L_n\stackrel{\mathcal{L}}{\to} B$ is not met either. }

\subsection{Proof of Theorem \ref{Theo3}}\label{AppendixA}
First, note that $\mathbb{E}[J_n]\to 0$ as $n\to\infty$, fulfilling \eqref{lim}. Then, by assumptions on the moments of $J^+_n$ and by definition of $\theta_n, f_n$, we have that
\begin{eqnarray*}
\lambda_n\mathbb{E}[J_n]&=&\mu-\lambda_n^{-1+c_1+c_2}=\mu+o(\lambda_n^{-1+c_1+c_2})\\
\lambda_n\mathbb{E}[J_n^2]&=&\sigma^2-\sigma^2\lambda_n^{-1+c_1}+\lambda_n^{-c_1}(-\lambda_n^{c_2}+\mu)^2=
\sigma^2+o(\lambda_n^{-1+c1})+o(\lambda_n^{-c_1+2c_2})\\
\lambda_n\mathbb{E}[J_n^4]&=&\lambda_n^{-c_3}-\lambda_n^{-1+c_1-c_3}+\lambda_n^{-3c_1}(-\lambda_n^{c_2}+\mu)^4=
o(\lambda_n^{-c_3})+o(\lambda_n^{-1+c_1-c_3})+o(\lambda_n^{-3c_1+4c_2}).
\end{eqnarray*}
Since $\lambda_n\to\infty$ as $n\to\infty$, {having $-1+c_1+c_2< 0, -1+c_1< 0, -c_1+2c_2< 0, -c3<0, -1+c_1-c_3<0$ and $-3c_1+4c_2<0$  guarantees that all infinitesimal terms in the expressions above vanish, implying that
conditions \eqref{EJn}-\eqref{EJ4n} are satisfied. This is achieved when $c_1\in(0,2/3),c_2 < c_1/2$ and $ c_3>0$. }

\subsection{Proof of Corollary \ref{cor2}}\label{AppendixD}
Looking at conditions \eqref{limnew}, we see that we now require $\lambda_n\mathbb{E}[J^+_n]\to\infty$, $\lambda_n\mathbb{E}[(J^+_n)^2]=\sigma^2$ and  $\lambda_n\mathbb{E}[(J^+_n)^4]=0$ as $n\to\infty$. Intuitively, from the limiting results reported in Table \ref{Table1} for the gamma {(for $\widetilde\sigma^2$ not linearly depending on $\mu$)}, inverse Gaussian and beta distribution (setting $\beta=\mu/\widetilde\sigma^2$), we  see that these conditions could be met  if $\mu\to\infty$, {as long as $\widetilde\sigma^2$ either does not depend on $\mu$ or, when it does, it remains constant as $\mu\to\infty$.} 
 In particular, one can easily prove that these distributions with $\mu={\lambda_n^{c_2}}$ fulfil the conditions in \eqref{limnew} in Theorem \ref{Theo3}.  Other random variables fulfilling these requirements are the generalised gamma and the beta prime distribution (results not shown). 
 \appendix

\section{}\label{AppA}
{When expressing $\mu$ and $\sigma^2$ as a function of $\lambda_n$ and of the parameters of the underlying jump distributions, looking at Table \ref{Table1}, we get
\begin{eqnarray*}
J_n\sim Bernoulli(p) &&
\mu=\sigma^2=\lambda_n p;\\
J_n \sim Poisson (\widetilde \lambda)&&
\mu=\sigma^2=\lambda_n\widetilde\lambda;\\
%J_n\sim \chi^2(k)&&\mu= \lambda_n k, \qquad \sigma^2=2\mu=2 \lambda_n k;\\
J_n\sim Gamma(\textrm{shape}\ \widetilde\alpha, \textrm{rate}\ \beta)&&\mu=\lambda_n \frac{\widetilde\alpha}{\beta}, \qquad \sigma^2=\lambda_n \frac{\widetilde\alpha(1+\widetilde\alpha)}{\beta^2};\\J_n \sim IG(\textrm{mean}\ \widetilde\mu, \textrm{shape}\ \widetilde\lambda)&&
\mu=\lambda_n \widetilde\mu, \qquad \sigma^2=\lambda_n \frac{\widetilde\mu^3}{\widetilde\lambda};\\
J_n\sim Beta(\textrm{shape}\ \widetilde\alpha, \textrm{shape}\ \beta)&&
\mu=\lambda_n\frac{\gamma}{\beta},\qquad \sigma^2=\lambda_n\frac{\gamma}{\beta(\beta+1)}.
\end{eqnarray*}
The results for $J_n\sim \chi^2(k)$ can be immediately recovered from the Gamma distribution by setting $\beta=1/2, k=2\widetilde\alpha$, and noting that $\sigma^2\approx 2\lambda_n\widetilde\alpha=2\lambda_n k$ as $\widetilde\alpha\to 0$ as $\lambda_n\to\infty$, see Table \ref{Table1}. }
}

{
\section{Limit L\'evy densities}\label{prooflevy}
When $J_n\sim Bernoulli(p)=Bernoulli(\mu/\lambda_n)$, a jump of amplitude one happens only when a success is observed, i.e. $\delta(x-1)$, with probability $\mu/\lambda_n$. Alternatively, a jump of amplitude zero, $\delta(x)$, is observed with probability $1-\mu/\lambda_n$. However, having null amplitude, it does not make the process jump. Thus
\[
\lambda_n f_{J_n}(x)=\lambda_n\left(\frac{\mu}{\lambda_n}\delta(x-1)+(1-\frac{\mu}{\lambda_n})\delta(x)\right)\to \mu\delta(x-1), \qquad \textrm{ as } n\to\infty,
\]
where we used the fact that $\lambda_n\delta (x)\to0$, since we have infinitely many jumps of zero amplitude. \\
When $J_n\sim Poi(\mu/\lambda_n)$,
\[
\lambda_n f_{J_n}(x)=\lambda_n\frac{\exp(-\mu/\lambda_n)(\mu/\lambda_n)^x}{x!} \to \mu\delta(x-1), \qquad \textrm{ as } n\to\infty,
\]
noting again that $\lambda_n\delta(x)\to 0$, as before, and $\lambda_n^{1-x}\to 0$ for $x=2,3,\ldots,$.\\
%When $J_n\sim \chi^2(k=\mu/\lambda_n)$, where $k$ denotes the degrees of freedom, using the recursive property of the gamma function $\Gamma(z+1)=z\Gamma(z)$, the fact that $\Gamma(1)=1$ and $\mu\lambda_n\to 0$ as $\lambda_n\to\infty$, we have
%\[
%\lambda_n f_{J_n}(x)=\lambda_n \frac{x^{\mu/(2\lambda_n)-1}e^{-x/2}}{2^{\mu/\lambda_n}\Gamma(\frac{\mu}{2\lambda_n})}=
%\lambda_n \frac{\frac{\mu}{2\lambda_n}x^{\mu/(2\lambda_n)-1}e^{-x/2}}{2^{\mu/\lambda_n}\Gamma(\frac{\mu}{2\lambda_n}+1)}\to 
%\frac{\mu}{2}x^{-1}e^{-x/2}, \qquad \textrm{ as } n\to\infty.
%\]
When $J_n\sim \Gamma(\widetilde\alpha,\beta)$ with shape parameter $\widetilde\alpha=\mu\beta/\lambda_n$ and rate parameter $\beta=\mu/\widetilde\sigma^2$, using the recursive property of the gamma function $\Gamma(z+1)=z\Gamma(z)$, the fact that $\Gamma(1)=1$, $\widetilde\alpha\to 0$ and $\lambda_n\widetilde\alpha\to\mu^2/\widetilde\sigma^2$ as $\lambda_n\to\infty$, we obtain 
\[
\lambda_n f_{J_n}(x)=\lambda_n\frac{\beta^{\widetilde\alpha}}{\Gamma(\widetilde\alpha)}x^{\widetilde\alpha-1}e^{-\beta x}
=\lambda_n\frac{\widetilde\alpha \beta^{\widetilde\alpha}}{\Gamma(\widetilde\alpha+1)}x^{\widetilde\alpha-1}e^{-\beta x}
\to \frac{\mu^2}{\sigma^2}x^{-1}e^{-\mu x/\widetilde\sigma^2}, \qquad \textrm{ as } n\to\infty.
\]
When $\beta=1/2$, i.e. $\widetilde\sigma^2=2\mu$, we recover the results for $J_n\sim \chi^2(k)$, with $k=\mu/(2\lambda_n)$.\\
When $J_n\sim IG(\widetilde\mu,\widetilde\lambda)$ with mean parameter $\widetilde\mu=\mu/\lambda_n$ and shape parameter $\widetilde\lambda=\mu^3/(\lambda_n^2\widetilde\sigma^2)$, we have that $\widetilde\mu\to 0$ and $\widetilde\lambda/\widetilde\mu^2\to\mu/\widetilde\sigma^2$. Therefore,
\begin{eqnarray*}
\lambda_n f_{J_n}(x)&=&\lambda_n\sqrt{\frac{\widetilde\lambda}{2\pi x^3}}\exp\left(-\frac{\widetilde\lambda(x-\widetilde\mu)^2}{2\widetilde\mu^2x}\right)=\sqrt{\frac{\mu^3}{2\pi\widetilde\sigma^2 x^3}}\exp\left(-\frac{\mu(x-\widetilde\mu)^2}{2\widetilde\sigma^2 x}\right)\\
&\to&\sqrt{\frac{\mu^3}{2\pi\widetilde\sigma^2 x^3}}\exp\left(-\frac{\mu x}{2\widetilde\sigma^2}\right), \qquad \textrm{ as } n\to\infty.
\end{eqnarray*}
Finally, when $J_n\sim Beta(\widetilde\alpha,\beta)$ with shape parameters $\widetilde\alpha=\mu \beta/\lambda_n$ and $\beta$, expressing the beta function as a function of the gamma function and using again the previous recursive formula, we have
\[
B(\widetilde\alpha,\beta)=\frac{\Gamma(\widetilde\alpha)\Gamma(\beta)}{\Gamma(\widetilde\alpha+\beta)}=\frac{\Gamma(\widetilde\alpha+1)\Gamma(\beta)}{\widetilde\alpha\Gamma(\widetilde\alpha+\beta)}.
\]
As $n\to\infty$, $\lambda_n\to\infty, \widetilde\alpha\to 0, \Gamma(\widetilde\alpha+1)\to 1, \Gamma(\widetilde\alpha+\beta)\to \Gamma(\beta)$ and 
 $\lambda_n /B(\widetilde\alpha,\beta)\to \mu\beta$, leading to
\[
\lambda_n f_{J_n}(x)= \lambda_n \frac{x^{\widetilde\alpha-1}(1-x)^{\beta-1}}{B(\widetilde\alpha,\beta)}\to\mu\beta x^{-1}(1-x)^{\beta}, \qquad \textrm{as } n\to\infty.
\]
}
\section*{Acknowledgements}
The authors would like to thank Martin Jacobsen for a fruitful correspondence on L\'evy processes. {The authors are grateful to Yan Qu for having provided us the Matlab codes for the exact simulation of the OU-Gamma and OU-IG processes.} 
This work was supported by the Austrian Exchange Service (OeAD-GmbH), bilateral project
CZ 19/2019 and by the Austrian Science Fund (FWF), project Nr. I 4620-N.

\end{document}